\documentclass[a4paper,10pt]{amsart}

\usepackage[latin1]{inputenc}
\usepackage{amsmath,amsfonts,amssymb,amsthm,url,geometry}

\newcommand{\e}{\varepsilon}
\renewcommand{\phi}{\varphi}
\newcommand{\iy}{\infty}

\renewcommand{\leq}{\leqslant}
\renewcommand{\geq}{\geqslant}

\newcommand{\N}{\mathbf{N}}
\newcommand{\Z}{\mathbf{Z}}

\newcommand{\R}{\mathbf{R}}

\newcommand{\E}{\mathbf{E}}
\renewcommand{\P}{\mathbf{P}}

\DeclareMathOperator{\vol}{vol}
\DeclareMathOperator{\card}{card}

\theoremstyle{plain}
\newtheorem{thm}{Theorem} 
\newtheorem*{thm-nonumber}{Theorem}
\newtheorem*{prop-nonumber}{Proposition}
\newtheorem*{claim*}{Claim}

\newtheorem{lem}{Lemma}

\theoremstyle{definition}

\theoremstyle{remark}

\begin{document}

\title{Maximal inequality for high-dimensional cubes}
\author{Guillaume \textsc{AUBRUN}}
\subjclass{42B25}
\keywords{Maximal inequality, high-dimensional cubes, Brownian bridge}

\begin{abstract}
We present lower estimates for the best constant appearing in the weak $(1,1)$ maximal inequality in the space $(\R^n,\|\cdot\|_{\iy})$. We show that this constant grows to infinity faster than $(\log n)^{1-o(1)}$ when $n$ tends to infinity. To this end, we follow and simplify the approach used by J.M.~Aldaz. The new part of the argument relies on Donsker's theorem identifying the Brownian bridge as the limit object describing the statistical distribution of the coordinates of a point randomly chosen in the unit cube $[0,1]^n$ ($n$ large).
\end{abstract}

\maketitle

\section*{Introduction}

Let $\vol(\cdot)$ denote the Lebesgue measure on $\R^n$. For $x \in \R^n$ and $r>0$, let $Q(x,r)$ denote the $n$-dimensional cube with center $x$ and edge length $2r$. For a positive Borel measure $\mu$ on $\R^n$, let $M\mu$ be the ``cubic'' centered Hardy--Littlewood maximal function
\[ M\mu(x) := \sup_{r > 0} \frac{\mu(Q(x,r))}{\vol(Q(x,r))} .\]
The maximal function is a fundamental tool in harmonic analysis, and it satisfies the following weak $(1,1)$ inequality: for any positive Borel measure $\mu$ on $\R^n$ and any $L > 0$,
\begin{equation} \label{ineqmax}
 L \cdot \vol \{ M\mu > L \} \leq C \mu(\R^n) .
\end{equation}
We denote by $\Theta_n$ the best possible $C$ appearing in \eqref{ineqmax}. Using a mollifying argument, one can check that restricting to absolutely continuous measures---or functions in $L^1(\R^n)$---does not alter the value of~$\Theta_n$.  

We are interested in the asymptotic behavior of $\Theta_n$. The easiest proof of \eqref{ineqmax} goes through Vitali's covering lemma, and gives $\Theta_n \leq 3^n$ (see for example \cite{rudin}, Chapter 7). This was greatly improved by Stein and Strömberg, who obtained $\Theta_n \leq C n \log n$ for some absolute constant $C$, using a more sophisticated covering argument  \cite{st-st}. This is the best known upper bound. 

Conversely, Aldaz proved recently that the sequence $(\Theta_n)$ tends to $+\iy$ when $n$ increases \cite{aldaz}. Aldaz's argument involves some tedious computations. It was shown in \cite{cs} that a careful look at the argument gives the lower bound $\Theta_n \geq c \sqrt{\log n}/\log \log n$ for some absolute constant $c>0$. 

We introduce in Aldaz's proof an extra tool: the Brownian bridge as a limit object from statistics. This has two consequences. First, this gives a shorter and conceptually simpler proof of the fact that $\Theta_n \to \infty$. Second, using extra known estimates on the Brownian bridge, we get a better lower bound on $\Theta_n$. This is our main theorem.

\begin{thm}
\label{main}
For any $0<\e<1$, there is a constant $c(\e) >0$ so that
\[ \Theta_n \geq c(\e) (\log n)^{1-\e} .\]
\end{thm}

We refer to \cite{aldaz} for additional background and a complete account of the bibliography. Let us just emphasize two related open problems:
\begin{enumerate} 
 \item Let $\tilde{M}\mu$ be the ``Euclidean'' maximal function, defined similarly to $M\mu$ but using Euclidean balls instead of
 cubes. Let $\tilde\Theta_n$ be the constant in the corresponding weak $(1,1)$ inequality. Is $\tilde\Theta_n$ bounded by an
 absolute consant ?
 \item For $p>1$, let $\Theta_n^{(p)}$ be the best constant in the strong $(p,p)$ inequality in $(\R^n,\|\cdot\|_\iy)$, i.e. the 
smallest $C$ so that  $\|Mf\|_{L^p} \leq C\|f\|_{L^p}$ for any function $f \in L^p(\R^n)$. Is $\Theta_n^{(p)}$ bounded by a constant depending only on $p$ ?
This has been answered affirmatively by Bourgain for $p>3/2$ (\cite{bourgain}, see also \cite{carbery}).
\end{enumerate}

A very closely related result due to Stein asserts that, for every $p>1$, the {\itshape Euclidean} maximal function $\tilde{M}$ does satisfy a strong $(p,p)$ inequality with constant independent of the dimension (see e.g. the Appendix of \cite{st-st} for a proof).

High-dimensional phenomena often involve probabilistic considerations, and indeed it is the case here. 
Let us sketch our proof. We follow closely Aldaz's approach, but using more sophisticated tools, so that we obtain more precise results.
To give a lower bound on $\Theta_n$, we choose $\mu$ to be a very simple measure: the counting measure on the lattice $\Z^n$ (after a suitable truncation). The key fact is that the value of $M\mu(x)$ is closely related to the statistical distribution of the coordinates of $x$ modulo $1$. By a theorem due to Donsker, the asymptotic statistical behavior is governed by a stochastic process $(\beta_t)_{0 \leq t \leq 1}$ called the Brownian bridge. More precisely, the typical value of the maximal function is related to the following quantity: 
\[ \sup_{\e \leq t \leq 1-\e} \frac{\beta_t}{\sqrt{t(1-t)}} .\]
The order of magnitude of the latter is given by the law of the iterated logarithm. We also need a quantitative estimate on the speed of convergence in Donsker's theorem, and to this effect we use the Koml\'os--Major--Tusn\'ady theorem.

\medskip

\noindent {\bf Acknowledgements:} I thank Nadine Guillotin-Plantard for helpful explanations about the Brownian bridge, and Jes\'us Mun\'arriz Aldaz for several comments on the paper.

\section{Qualitative approach}

In this section, we introduce some objects needed to prove our theorem, and give a proof of Aldaz's result ($\Theta_n \to \iy$) which requires less calculations that the original one.
Throughout the paper, $\mu$ will denote the measure on $\R^n$ obtained by putting a Dirac mass on each lattice point $x \in \Z^n$. That is, for any Borel set $A$, $\mu(A) := \card(A \cap \Z^n)$. 

\subsection{Reduction to $[0,1]^n$}
Although the measure $\mu$ has infinite mass, it can be used to estimate the constant $\Theta_n$, as showed by the following lemma
\begin{lem}
\label{lemma-truncation}
For any $L > 1$, 
\[L \cdot \vol \left(\left\{ x \in [0,1]^n \textnormal{ s.t. } M\mu(x) > L \right\}\right) \leq \Theta_n .\]
\end{lem}

\begin{proof}[Proof of lemma \ref{lemma-truncation}]
The point is that large cubes can be ignored when computing $M\mu$. Indeed, since a cube of edge length $2r$ contains at most $(2r+1)^n$ lattice points, we get for any $L>1$,
\begin{equation} \label{largecubes} \sup_{r > n/2\log L} \frac{\mu(Q(x,r))}{\vol(Q(x,r))} \leq  
\sup_{r > n/2\log L} \left( 1 +\frac{1}{2r} \right)^n \leq 
\left( 1 + \frac{\log L}{n} \right)^n \leq L. \end{equation}
Let $\mu_R$ be the restriction of $\mu$ to the cube $Q(0,R)$. It follows from \eqref{largecubes} that for any $x \in 
Q(0,R-n/2\log L)$, $M\mu(x) >L$ if and only if $M\mu_R(x) >L$. Using the obvious fact that $M\mu$ is $\Z^n$-periodic,
this gives
\begin{eqnarray*} \Theta_n & \geq & \frac{L \cdot \vol\left(\left\{x \in \R^n \textnormal{ s.t. } M\mu_R(x) >L \right\} \right)}{\mu_R(\R^n)} \\ 
& \geq & \frac{L \cdot \vol \left(\{ x \in Q(0,R-n/2\log L) \textnormal{ s.t. } M\mu(x)>L \}\right) }{(2\lfloor R \rfloor +1)^n} \\
&\geq & \left(\frac{2 \lfloor R-n/2\log L \rfloor}{2\lfloor R \rfloor+1}\right)^n \cdot L \cdot \vol \left(\left\{ x \in [0,1]^n \textnormal{ s.t. } M\mu(x) > L \right\} \right). 
\end{eqnarray*}

Ir remains to take $R \to \iy$.
\end{proof}

\subsection{Maximal function and statistical distributions of coordinates} 
For $x \in [0,1]^n$, we relate the value of $M\mu(x)$ to the statistical distribution of the coordinates of $x$
in the interval $[0,1]$.
For $t \in (0,1)$, a number $x \in [0,1]$ is called $t$-centered if it belongs to $\left[\frac{1-t}{2},\frac{1+t}{2}\right]$. Define the following set

\[ E_{t,K}^n := \left\{ x \in [0,1]^n \textnormal{ with at least } nt+K\sqrt{n t(1-t)} \textnormal{ $t$-centered coordinates} \right\}. \]

This definition may look strange but the motivation comes from probability theory. Think of $x$ as a random variable uniformly distributed on
$[0,1]^n$. The number of $t$-centered coordinates is a random variable with expectation $nt$, and standard deviation $\sqrt{nt(1-t)}$. In particular, it follows from the central limit theorem that
\[ \lim_{n \to \iy} \vol(E_{t,K}^n) = \P(G > K), \]
where $G$ is a standard Gaussian random variable. The next lemma shows that, asymptotically, the maximal function is large on the set $E_{t,K}^n$. The important point is that the lower bound depends only on $K$, and not on $t$. This lemma is essentially a variant on Claim 1 from \cite{aldaz}.

\begin{lem}
\label{lemme1}
For any $\eta>0$, there exists a constant $D(\eta)$ so that for any $K>0$ and  $t \in (0,1)$, if $n \geq \frac{D(\eta)K^2}{t(1-t)}$, then 
\[ E_{t,K}^n \subset \left\{ M\mu > e^{K^2/(2+\eta)} \right\} . \label{lem1} \] 
\end{lem}

Since the proof of lemma \ref{lemme1} is independent from the rest of the argument, we postpone it to the
end of the paper. 

\subsection{Probabilistic point of view}

If we think of $([0,1]^n, \vol)$ as a probability space, the coordinates $x_1,\dots,x_n$ are independent random variables. Let $X_i := 2|x_i-\frac{1}{2}|$. It is easily checked that $(X_1,\dots,X_n)$ are also independent random variables uniformly distributed on $[0,1]$. For $0 \leq t \leq 1$, we introduce new random variables which count the number of $t$-centered coordinates
\begin{equation} \label{def-alpha} \alpha_t^{(n)} := \frac{1}{\sqrt{n}} \sum_{i=1}^n \left( {\bf 1}_{\left\{\frac{1-t}{2} \leq x_i \leq \frac{1+t}{2} \right\}} -t \right) 
= \frac{1}{\sqrt{n}} \sum_{i=1}^n \left( {\bf 1}_{\{X_i \leq t \}} -t \right)
.\end{equation}
The Lebesgue measure of the union (over $t$) of the sets $E^n_{t,k}$ become the probability that a certain supremum exceeds $K$. For example, for any $0 < \e < 1/2$ , 
\begin{equation} \label{ana-prob} \vol \left( \bigcup_{\e \leq t \leq 1-\e} E_{t,K}^n \right) = \P \left( \sup_{\e \leq t \leq 1-\e} \frac{\alpha_t^{(n)}}{\sqrt{t(1-t)}} \geq K \right). \end{equation}
(it is easily checked that the set involved is measurable).

\subsection{A simple proof that $(\Theta_n)$ tends to $+\iy$}

We show that the sequence $(\Theta_n)$ is unbounded using qualitative arguments. We actually show a more precise result: the typical value of the maximal function is large.

\begin{thm}
\label{theo2}
For any $L > 0$, we have
\[ \lim_{n \to \iy} \vol \left( \{ x \in [0,1]^n \textnormal{ s.t. } M\mu(x) > L \} \right) = 1\]
\end{thm}
It is clear that Theorem \ref{theo2} and Lemma \ref{lemma-truncation} imply together that the sequence $(\Theta_n)$ tends to $+\iy$. Before passing to the proof of Theorem \ref{theo2}, let us state a lemma describing the joint asymptotic behaviour of $(\alpha_t^{(n)})_{n \in \N}$ when $t$ takes finitely many values.

\begin{lem}
\label{clt}
Let $t_1,\dots,t_N$ be elements of $[0,1]$. The random vector $(\alpha_{t_1}^{(n)},,\dots,\alpha_{t_N}^{(n)})$ converges weakly towards the Gaussian random vector $(\beta_{t_1},\dots,\beta_{t_N})$ given by the covariance
\[ \E \beta_{t_i} \beta_{t_j} = t_i(1-t_j) \ \ \ \ \ \ \textnormal{ for } 0 \leq t_i \leq t_j \leq 1. \]
\end{lem}

\begin{proof}
One computes the covariance of the random variables involved in the definition of $(\alpha_t^{(n)})$
\[ \E \left[ ( {\bf 1}_{\{X \leq t\}} - t) ( {\bf 1}_{\{X \leq u\}} -u ) \right] = \min(t,u) - tu .\]
The lemma now follows from the multivariate central limit theorem (see \cite{durrett}, p.168).
\end{proof}

\begin{proof}[Proof of Theorem \ref{theo2}]
Fix $K$ and $\e >0$. We apply Lemma \ref{clt} with $t_k=\frac{k}{k+1}$ and $N$ to be chosen. Having \eqref{ana-prob} in mind, we obtain the following
\begin{eqnarray*} \ell_N := \lim_{n \to \iy} \vol \left( \bigcup_{k=1}^N E_{t_k,K}^n \right)  & = & \lim_{n \to \iy} \P \left( \max_{1 \leq k \leq N} \frac{\alpha_{t_k}^{(n)}}{\sqrt{t_k(1-t_k)}} >K \right) \\ &=& \P \left( \max_{1 \leq k \leq N} \frac{\beta_{t_k}}{\sqrt{t_k(1-t_k)}} >K \right) 
\end{eqnarray*}
The last equality follows from Lemma \ref{clt}. Let $(G_n)$ be a sequence of independent $N(0,1)$ random variables, and let $B_n=G_1+\dots+G_n$. The sequence $(B_n)$ is actually a Brownian motion restricted to integer times. We claim that the following random vectors coincide in distribution
\begin{equation} \label{bridge-motion} \left(\frac{\beta_{t_k}}{\sqrt{t_k(1-t_k)}}\right)_{1 \leq k \leq N} \sim \left( \frac{B_k}{\sqrt{k}} \right)_{1 \leq k \leq N}. \end{equation}
Indeed, since both are centered Gaussian vectors, it is enough to show that they share the same covariance, as it is easily checked: for $i \leq j$,
\[ \E \left( \frac{\beta_{t_i}}{\sqrt{t_i(1-t_i)}} \frac{\beta_{t_j}}{\sqrt{t_j(1-t_j)}}  \right) = \E \left( \frac{B_i}{\sqrt{i}} \frac{B_j}{\sqrt{j}} \right) = \sqrt{\frac ij} .\]
We obtain therefore
\[ \ell_N = \P \left( \max_{1 \leq k \leq N} \frac{B_k}{\sqrt{k}} > K \right) \ \ \ \ \ \textnormal{ and } \ \ \ \ \ \ \lim_{N \to \iy} \ell_N  = \P \left(\sup_{k \in \N^*} \frac{B_k}{\sqrt{k}} > K \right).\]
The next claim is a well-known property of Brownian motion.

\begin{claim*}
With $(B_k)_{k \in \N^*}$ as before, we have for any $K>0$, $\displaystyle \P \left( \limsup_{k \to \iy} \frac{B_k}{\sqrt{k}} >K \right) =1$.
\end{claim*}
Since $\lim \ell_N=1$, we can choose $N$ to that $\ell_N > 1-\e/2$. This implies that for large enough dimensions,
\[ \vol \left( \bigcup_{k=1}^N E_{t_k,K}^n \right) \geq 1-\e .\]
Using Lemma \ref{lemme1} with $\eta=1$, we conclude that for large enough dimensions, the maximal function $M\mu$ is larger that $\exp(K^2/3)$ on the above set. This proves Theorem \ref{theo2} (with $\exp(K^2/3)$ instead of $L$). 
\end{proof}

\begin{proof}[Proof of the Claim]
It follows from Kolmogorov's 0/1 law (\cite{durrett}, p.61) that the event into consideration has probability 0 or 1. By Fatou's lemma
\[ \P \left( \limsup_{k \to \iy} \frac{B_k}{\sqrt{k}} >K \right) \geq \limsup_{k \to \iy} \P\left(\frac{B_k}{\sqrt{k}} >K \right) .\]
The quantity in the r.h.s. does not depend on $k$ and is nonzero. This proves the claim.
\end{proof}

\section{Proof of theorem \ref{main}}

To keep the previous section elementary, we did not mention about stochastic processes. We now introduce the material needed.

\subsection{The Brownian bridge}
A theorem by Donsker \cite{donsker} asserts that when $n$ tends to $+\iy$, the process $(\alpha_t^{(n)})_{0 \leq t \leq 1}$ defined in \eqref{def-alpha} converges in the supremum norm towards a Brownian bridge $(\beta_t)_{0 \leq t \leq 1}$. Recall that a Brownian bridge $(\beta_t)_{0 \leq t \leq 1}$ is defined
as a Gaussian process which is almost surely continuous and given by the covariance 
\[ \E \beta_t \beta_u = t(1-u) \ \ \ \ \ \ \textnormal{ for } 0 \leq t \leq u \leq 1. \]
In particular, $\beta_0=\beta_1=0$ almost surely. We refer to (\cite{durrett}, Chapter 7.8) for more information on the Brownian bridge. Note that Lemma \ref{clt} from previous section deals with (elementary) convergence of finite-dimensional marginals; however it will be convenient to consider infinitely many values of $t$ simultaneously.

Since we are interested in non-asymptotic bounds, we also need more quantitative results about the speed of convergence in Donsker's theorem. This is provided by the following theorem, due to Koml\'os, Major and Tusn\'ady \cite{kmt}. We use the version appearing in a paper by Bretagnolle and Massart~\cite{bm}.

\begin{thm-nonumber}[Koml\'os--Major--Tusn\'ady]
For any $n \geq 1$, there exists a probability space on which are defined 
\begin{itemize}
\item A $n$-tuple of i.i.d. random variables uniformly distributed on $[0,1]$: $(X_1,\dots,X_n)$,
\item a Brownian bridge: $(\beta_t^{(n)})_{0 \leq t \leq 1}$,
\end{itemize}
so that, denoting
\[ \alpha_t^{(n)} = \frac{1}{\sqrt{n}} \sum_{j=1}^n \left( {\bf 1}_{\{X_j \leq t\}} -t \right), \]
the following inequality is valid for any $x>0$:
\[ \P \left( \sup_{0 \leq t \leq 1} | \sqrt{n}(\alpha_t^{(n)}-\beta_t^{(n)})| > 12 \log n + x \right) < 2 \exp(-x/6). \]
\end{thm-nonumber}

\subsection{Relation to Brownian motion}

There are several ways to relate the Brownian bridge to the Brownian motion. Let $(B_t)_{t \geq 0}$ be a standard Brownian motion, given by the covariance $\E B_tB_u = \min(t,u)$. It is easily checked, just by computing the covariance, that the process $(B_t-tB_1)_{0 \leq t \leq 1}$ is a Brownian bridge. Similarly, the process
\begin{equation} \label{bb-bm} \left((1-t)B_{\frac{t}{1-t}}\right)_{0 \leq t \leq 1} \end{equation}
is also a Brownian bridge --- we actually used this transformation in the proof of Theorem \ref{theo2}. 

As the formula \eqref{ana-prob} hints, we need to control the supremum of $\beta_t/\sqrt{t(1-t)}$. Using the transformation given by \eqref{bb-bm}, this amounts to controlling the supremum of $B_t/\sqrt{t}$, where $(B_t)$ is a Brownian motion. In the previous section, we used the well-known fact that this supremum, when taken over $(0,\iy)$, is almost surely $+\iy$. To obtain concrete lower bounds, we need a more quantitative result, given by the law of the iterated logarithm. This is the statement of the next lemma. Since we could not find this exact statement in the literature, we include a proof at the end of the paper.

\begin{lem}
\label{lemma-LIL}
Let $(\beta_t)_{0 \leq t \leq 1}$ be a Brownian bridge. For any $\eta \in (0,2)$, there exists a constant $c(\eta)>0$ so that for any $0 < \e \leq 1/e$,
\begin{equation} \label{lil-bb}
 \P \left( \sup_{\e \leq t \leq 1-\e} \frac{\beta_t}{\sqrt{t(1-t)}} \geq \sqrt{(2-\eta)\log \log (1/\e)} \right) \geq c(\eta) .
\end{equation}
\end{lem}
We are now ready to prove our main theorem.

\subsection{Proof of Theorem \ref{main}}
Fix $\eta >0$ and let $c(\eta)$ the constant given by Lemma \ref{lemma-LIL}, which we apply with the choice $\e := \log^2 n/n$. Choose now $x>0$ so that $2 \exp(-x/6) < c(\eta)/2$.  Applying Koml\'os--Major--Tusn\'ady theorem, we obtain a coupling of $\alpha_t^{(n)}$ and $\beta_t^{(n)}$ so that the following events hold simultaneously with probability larger than $c(\eta)/2$
\[ \left\{ \begin{array}{c} \displaystyle \sup_{\e \leq t \leq 1-\e} \frac{\beta_{t}^{(n)}}{\sqrt{t(1-t)}} \geq \sqrt{(2-\eta)\log \log (1/\e)} 
            \\
\displaystyle \forall t \in [0,1], \ \ \ \ \ \alpha_t^{(n)}> \beta_t^{(n)} - \frac{12 \log n + x}{\sqrt{n}}
           \end{array}
\right. \]
This shows that
\[ \P \left( \sup_{\e \leq t \leq 1-\e}  \frac{\alpha_{t}^{(n)}}{\sqrt{t(1-t)}} \geq \sqrt{(2-\eta)\log \log (1/\e)} 
- \frac{12 \log n + x}{\sqrt{n\e(1-\e)}} \right) \geq \frac{c(\eta)}{2}.\]
Set $K := \sqrt{(2-\eta)\log \log (1/\e)} - \frac{12 \log n + x}{\sqrt{n\e(1-\e)}}$. Using formula \eqref{ana-prob}, we obtain
\[ \vol \left( \bigcup_{\e \leq t \leq 1-\e} E_{t,K}^n \right) \geq \frac{c(\eta)}{2} .\]
One checks that for $n$ large enough, $K \geq \sqrt{(2-2\eta) \log \log n}$. Also, for $n$ large enough, $n \geq \frac{D(\eta) K^2}{\e(1-\e)}$, so that we can use lemma \ref{lem1} and conclude that
\[ \vol \left(\left\{ x \in [0,1]^n \textnormal{ s.t. } M\mu(x) > e^{K^2/(2+\eta)} \right\}\right) \geq \frac{c(\eta)}{2}. \]
Using lemma \ref{lemma-truncation}, this implies that for $n$ large enough,
\[ \Theta_n \geq \frac{c(\eta)}{2} e^{K^2/(2+\eta)} \geq \frac{c(\eta)}{2} (\log n)^{\frac{2-2\eta}{2+\eta}}. \]
This inequality can be extended to all $n$ by adjusting $c(\eta)$ if necessary. Since $\frac{2-2\eta}{2+\eta}$ is arbitrarily close to 1, this proves the theorem.

\section{Proof of lemmas \ref{lemme1} and \ref{lemma-LIL}}

\subsection{Proof of lemma \ref{lemme1}}
A variant of this lemma appears in \cite{aldaz}. Let $x \in E_{t,K}^n$ and define $D$ by the relation $n=\frac{D K^2}{t(1-t)}$. We will show that if $D > 9$, then
\begin{equation} \label{aaaa} \log M\mu(x) \geq \frac{K^2}{2} - \frac{K^2}{6 \sqrt{D}} - \frac{4K^2 \sqrt{D}}{(\sqrt{D}-3)^3},  
\end{equation}
The lemma follows since the right-hand side of \eqref{aaaa} is larger than $\frac{K^2}{2+\eta}$ for $D$ large enough. We start with an elementary one-dimensional consideration: if $x_i \in [0,1]$, then
\[
\card \left( \left[x_i-\left(s-\frac{1-t}{2}\right),x_i+\left(s-\frac{1-t}{2}\right)\right] \cap \Z \right) = 
\begin{cases}
2s & \textnormal{ if $x_i$ is $t$-centered}, \\
2s-1 & \textnormal{ otherwise.}
\end{cases}
\]
Consequently,  if $x \in [0,1]^n$ has at least $m$ $t$-centered coordinates, then 
\[ M\mu(x) \geq \sup_{s \in \N^*} \frac{\mu(Q(x,s-(1-t)/2))}{\vol(Q(x,s-(1-t)/2))} = \sup_{s \in \N^*} \frac{(2s)^m(2s-1)^{n-m}}{(2s-(1-t))^n} \]
(we do not assume that $m$ is an integer). If $x \in E_{t,K}^n$, we may choose $m = nt + K \sqrt{nt(1-t)}$ and therefore
\[ \log M\mu(x) \geq \sup_{s \in \N^*} F(s) ,\]
where $F$ is the function defined as
\begin{eqnarray*} F(s) & = & K^2 \left[ \left( \frac{D}{1-t}+\sqrt{D} \right) \log(2s) + \left( \frac{D}{t} - \sqrt{D} \right) \log(2s-1) - \frac{D}{t(1-t)} \log(2s-1+t) \right] \\
 & = & K^2 \left[ \left( \frac{D}{1-t}+\sqrt{D} \right) \log \left(\frac{2s}{2s-1+t}\right) + \left( \frac{D}{t} - \sqrt{D}
 \right) \log \left( \frac{2s-1}{2s-1+t} \right) \right] .
\end{eqnarray*}
We first compute the supremum of $F$ over $s \in \R^+$. One gets the following expression for the derivative:
\[ F'(s) = \frac{K^2 \sqrt{D} ( \sqrt{D} +1-t-2s)}{s(2s-1)(2s-1+t)}.\]
Thus, $F$ is maximal on $\R^+$ at the point $s_0$ defined as
\[ s_0 := \frac{\sqrt{D}+1-t}{2} > 1.\]
The maximal value of $F$ is 
\[ F(s_0) = K^2 \left( \frac{D}{1-t} + \sqrt{D} \right) \log \left( 1 + \frac{1-t}{\sqrt{D}} \right) + K^2 \left( \frac{D}{t} - \sqrt{D} \right) \log \left( 1 - \frac{t}{\sqrt{D}} \right).\]
Let $\Phi(x) = (1+x) \log(1+x)$. One checks by consecutive differentiation that $\Phi(x) \geq x+\frac{x^2}{2}-\frac{x^3}{6}$ for $x > -1$. We get
\[ F(s_0) = \frac{K^2D}{1-t} \Phi\left(\frac{1-t}{\sqrt{D}}\right) + \frac{K^2D}{t} \Phi\left(\frac{-t}{\sqrt{D}}\right) \geq \frac{K^2}{2} + \frac{K^2(2t-1)}{6\sqrt{D}}  .\]
However, since we are only allowed to consider integer-valued $s$, we need to
evaluate $F(\lfloor s_0 \rfloor)$, and show that it is not very different from $F(s_0)$. For $z \in [s_0-1,s_0]$, we have
\[ 2z \geq 2z-1+t \geq 2z-1 \geq 2s_0-3 \geq \sqrt{D} - 3 > 0.\]
This implies, for $z \in [s_0-1,s_0]$
\[ | F'(z) | \leq \frac{4K^2D}{(\sqrt{D}-3)^3} .\]
Applying the mean value theorem to $F$ between $\lfloor s_0 \rfloor$ and $s_0$ gives 
\[
\log M\mu(x) \geq F(\lfloor s_0 \rfloor) \geq \frac{K^2}{2} + \frac{K^2(2t-1)}{6\sqrt{D}} - \frac{4K^2D}{(\sqrt{D}-3)^3}.
\]
Since $2t-1 > -1$, we get \eqref{aaaa}, as we claimed.

\subsection{Proof of lemma \ref{lemma-LIL}}

Fix $\eta \in (0,2)$ and let $(B_t)_{t\geq 0}$ be a standard Brownian motion. Using the transformation given by \eqref{bb-bm}, we check that, denoting $A = 1/\e$,
\[ \P \left( \sup_{\e \leq t \leq 1-\e} \frac{\beta_t}{\sqrt{t(1-t)}} \geq \sqrt{(2-\eta)\log \log (1/\e)} \right) = \P \left( \sup_{\frac{1}{A-1} \leq t \leq A-1} \frac{B_t}{\sqrt{t}} \geq \sqrt{(2-\eta)\log \log A} \right)  .   \]
We will actually estimate the supremum of $B_t/\sqrt{t}$ over $[1,A-1]$. Since $(B_t)$ and $(-B_t)$ have the same distribution, we have
\[  \P \left( \sup_{1 \leq t \leq A-1} \frac{B_t}{\sqrt{t}} \geq \sqrt{(2-\eta)\log \log A} \right) \geq \frac{1}{2}
  \P \left( \sup_{1 \leq t \leq A-1} \frac{|B_t|}{\sqrt{t}} \geq \sqrt{(2-\eta)\log \log A}  \right) \]
Choose $\alpha > 1$ large enough so that $\rho<2$, where $\rho$ denotes the number
\[ \rho := \frac{(2-\eta)(1+\sqrt{\alpha})^2}{\alpha-1} .\] Let $N:=\left\lfloor \frac{\log (A-1)}{\log \alpha} \right\rfloor-1$ and for $1 \leq j \leq N$, consider the event
\[ E_j:= \left\{ \omega \textnormal{ s.t. } \frac{B_{\alpha^{j+1}}(\omega) - B_{\alpha^j}(\omega)}{\sqrt{\alpha^{j+1}-\alpha^j}} \geq \sqrt{\rho\log \log A} \right\} . \]
One checks that if $E_j$ holds, then at least one of the following inequalities is true:
\[ \begin{cases}  \textnormal{either } & B_{\alpha^{j+1}} \geq \sqrt{(2-\eta) \log \log A} \cdot \sqrt{\alpha^{j+1}} \\  
 \textnormal{or } &  B_{\alpha^j} \leq - \sqrt{(2-\eta) \log \log A} \cdot \sqrt{\alpha^{j}}. \\ \end{cases} \]
Consequently,
\[ \P \left( \sup_{1 \leq t \leq A} \frac{|B_t|}{\sqrt{t}} \geq \sqrt{(2-\eta)\log \log A} \right) \geq  \P \left( \bigcup_{1 \leq j \leq N} E_j \right) .\]
Now because the Brownian motion has independent increments, the events $E_j$ are independent, and they all have the same probability
\[ \P(E_j) = \P \left(G > \sqrt{\rho \log\log A} \right) ,\]
where $G$ is a standard Gaussian random variable. We use the following estimate (see \cite{durrett}, p.6): for $x \geq 2$
\[ \P(G > x ) \geq \frac{x^{-1}-x^{-3}}{\sqrt{2\pi}} e^{-x^2/2} \geq C(\rho) \exp(-x^2/\rho) \]
for some constant $C(\rho)>0$. We obtain $\P(E_j) \geq C(\rho) (\log A)^{-1}$. Finally, since the events $E_j$ are independent,
\[ \P \left( \bigcup_{1 \leq j \leq N} E_j \right) \geq 1 - \left( 1 - \frac{C(\rho)}{\log A} \right)^N \geq 1-\exp \left(-\frac{C(\rho)N}{\log A} \right) . \]
Given our choice of $N$, one checks that the last expression is bounded below by some positive constant depending only on $\eta$ when $A$ tends to $+\iy$. This proves the lemma, at least for $A$ large enough. Small values of $A$ can be taken into account by adjusting a posteriori the constant $c(\eta)$ if necessary.

\noindent
Institut Camille Jordan, Université de Lyon 1.
\\ 
E-mail: \verb!aubrun@math.univ-lyon1.fr! 
\\ 
Homepage:  \verb!http://math.univ-lyon1.fr/~aubrun/!

\end{document}